\documentclass[a4paper,reqno]{amsart}%
\usepackage{amsfonts}
\usepackage{amsmath}
\usepackage{amssymb}
\usepackage{amscd}
\usepackage{graphicx}
\usepackage{fancyhdr}
\usepackage{color}
\usepackage{graphicx}
\usepackage{gloss}

\usepackage{mathtools}               

\usepackage[normalem]{ulem}

\usepackage{tikz}                  

\usepackage{graphicx} 
\usepackage{float} 
\usepackage{subfigure} 

\usepackage{algorithm}
\usepackage{algorithmic}
\usepackage{float}
\usepackage{lipsum}
\makeatletter
\newenvironment{breakablealgorithm}
{
	\begin{center}
		\refstepcounter{algorithm}
		\hrule height.8pt depth0pt \kern2pt
		\renewcommand{\caption}[2][\relax]{
			{\raggedright\textbf{\ALG@name~\thealgorithm} ##2\par}%
			\ifx\relax##1\relax 
			\addcontentsline{loa}{algorithm}{\protect\numberline{\thealgorithm}##2}%
			\else 
			\addcontentsline{loa}{algorithm}{\protect\numberline{\thealgorithm}##1}%
			\fi
			\kern2pt\hrule\kern2pt
		}
	}{
		\kern2pt\hrule\relax
	\end{center}
}
\makeatother

\allowdisplaybreaks[4]      

\input xy
\xyoption{all}

\newtheorem{theorem}{Theorem}

\newtheorem*{main theorem}{Main Theorem}

\newtheorem{case}{Case}
\newtheorem{claim}{Claim}
\newtheorem{claimproof}{Proof of claim}

\newtheorem{conjecture}[theorem]{Conjecture}

\newtheorem{lemma}[theorem]{Lemma}

\newtheorem{proposition}[theorem]{Proposition}

\begin{document}
\title[Lower bound for the simplicial volume]{Lower Bound for the Simplicial Volume of Closed Manifolds Covered by $\mathbb{H}^{2}\times\mathbb{H}^{2}\times\mathbb{H}^{2}$}

\author{Xiaofeng Meng}
\address{School of Mathematical Sciences, Fudan University, Shanghai 200433, China}
\email{xfmeng17@fudan.edu.cn}

\date{2021/3/14}
\begin{abstract}
We estimate the upper bound for the $\ell^{\infty}$-norm of the volume form on $\mathbb{H}^2\times\mathbb{H}^2\times\mathbb{H}^2$ seen as a class in $H_{c}^{6}(\mathrm{PSL}%
_{2}\mathbb{R}\times\mathrm{PSL}_{2}\mathbb{R}\times\mathrm{PSL}_{2}\mathbb{R};\mathbb{R})$. This %
gives the lower bound for the simplicial volume of closed Riemennian manifolds covered by $\mathbb{H}^{2}\times\mathbb{H}^{2}\times\mathbb{H}^{2}$. The proof of these facts yields an algorithm to compute the lower bound of closed Riemannian manifolds covered by $\big(\mathbb{H}^2\big)^n$.
\end{abstract}
\maketitle

\section{Introduction}

The simplicial volume is a topological invariant of manifolds introduced by Gromov (\cite{gromov1982volume}) and %
Thurston (\cite{thurston1978geometry}). For an oriented closed connected $n$-dimensional manifold $M$, $H_{n}(M,\mathbb{Z})$ is generated by the %
fundamental class $[M]_{\mathbb{Z}}$ of $M$. Denote the image of %
$[M]_{\mathbb{Z}}$ via the change of coefficients map $H_{n}(M,\mathbb{Z})\hookrightarrow H_{n}(M,\mathbb{R})$ by $[M]\in H_{n}(M,\mathbb{R})$. Then the simplicial volume of $M$ is defined by%

\[
\|M\|=\mathrm{inf} \left \{ \sum_{i}|a_{i}|\ \middle | \  [\sum_{i}a_{i}\sigma_{i}  ]=[M]\in %
H_{n} (M,\mathbb{R}  ) \right \}.
\]
Moreover, simplicial volume can also be defined for a non-orientable closed connected manifold $M$. Let $\widehat{M}$ be the oriented double covering of $M$. Then the simplicial volume of $M$ is defined by $\|M\|=\|\widehat{M}\|/2$. 

Although several vanishing and non-vanishing results for the simplicial volume are presented now, the exact value of %
non-vanishing simplicial volume has only been calculated for a few cases. Those include closed hyperbolic manifolds (\cite{gromov1982volume}, \cite{thurston1978geometry}), Hilbert modular surfaces (\cite{loh2007simplicial}) and closed manifolds covered by $\mathbb{H}^2\times\mathbb{H}^2$ (\cite{bucher2008simplicial}). 

For the simplicial %
volume of products,  we have the following result by \cite{gromov1982volume} for closed manifolds $M$ and $N$ with dimension $m$ and $n$, respectively.
\begin{equation}\label{equ: simplicial volume of product of manifolds}
    \|M\|\cdot\|N\| \leq \|M\times N\| \leq \binom{n+m}{n}\|M\|\cdot\|N\|.
\end{equation}
When $M$ and $N$ are closed %
surfaces, it has been proved in \cite{bucher2008simplicial} that
\begin{equation}\label{equ: simplicial volume of surfaces}
    \|M\times N\|=\frac{3}{2}\|M\|\cdot\|N\|.
\end{equation}
Combining (\ref{equ: simplicial volume of product of manifolds}) and (\ref{equ: simplicial volume of surfaces}), we get
\[
\dfrac{3}{2}\prod_{i=1}^3\|M_i\|\leq\Bigg\|\prod_{i=1}^3 M_i\Bigg\|\leq\dfrac{45}{2}\prod_{i=1}^3\|M_i\|,
\]
for closed surfaces $M_i$, $i=1,2,3$. 
In this paper we improve the above estimation of the lower bound to the following one and furthermore, offer an algorithm to compute the lower bound of closed Riemannian manifolds covered by $ (\mathbb{H}^2  )^n$.

\begin{theorem}
\label{Thm:product of surfaces}Let $M_{i}$ be a closed surface, where $i=1,2,3$. Then %
\[
\Bigg\|\prod_{i=1}^3M_{i}\Bigg\| \geq \frac{45}{11}\prod_{i=1}^3 \|M_i  \|.
\]
\end{theorem}

Recall the proportionality principle by Gromov in \cite{gromov1982volume}. Let $M$ be an $n$-dimensional closed Riemannian manifold. Then
\begin{equation}\label{equ: Gromov's proportionality principle}
    \|M\|=\dfrac{\mathrm{Vol(M)}}{\|[\omega_{\widetilde{M}}]_c^G\|_{\infty}},
\end{equation}
where $\widetilde{M}$ is the universal covering of $M$, $G$ is the group of orientation-preserving isometries of $\widetilde{M}$, $\omega_{\widetilde{M}}$ is the volume form of $\widetilde{M}$ and $[\omega_{\widetilde{M}}]_c^G$ is the volume class of $\widetilde{M}$ viewed as a class in $H^n(C_c^*(\widetilde{M})^G)$. According to \cite{thurston1978geometry}, the constant $\|[\mathrm{Vol}_{\widetilde{M}}]_c^G\|_{\infty}$ is $\pi$ for a closed oriented surface $M$ supporting a hyperbolic structure. According to \cite{gromov1982volume}, the simplicial volume of a closed connected manifold admitting a self-map of non-trivial degree (i.e., not equal to $-1$, $0$, or $1$) vanishes. This implies that the simplicial volume of a closed surface covered by $2$-sphere or torus vanishes. Therefore, Theorem \ref{Thm:product of surfaces} naturally holds when one of the universal coverings of $M_1$, $M_2$ and $M_3$ is $2$-sphere or Euclidean plane. To prove Theorem \ref{Thm:product of surfaces}, we only need to consider the case where the universal covering of $M_1\times M_2\times M_3$ is $\mathbb{H}^2\times\mathbb{H}^2\times\mathbb{H}^2$. Hence the following Theorem proves Theorem \ref{Thm:product of surfaces}.
\begin{theorem}
\label{Thm: simpl vol covered by HxHxH}Let $M$ be a closed Riemannian manifold whose universal covering is $\mathbb{H}^{2}\times\mathbb{H}^{2}\times\mathbb{H}^{2}$. Then 

\[
 \Vert M  \Vert
 \geq \frac{45}{11\pi^3}\mathrm{Vol} (M  ).%
\]

\end{theorem}

To prove Theorem \ref{Thm:product of surfaces} and \ref{Thm: simpl vol covered by HxHxH}, we recall a way to compute the proportionality constant $\frac{\|M\|}{\mathrm{Vol}(M)}$ provided by Bucher-Karlsson in \cite{bucher2008simplicial} for closed locally symmetric spaces of noncompact type. Let $M$ be an $n$-dimensional closed locally symmetric space of non-compact type. We use the same notations as in (\ref{equ: Gromov's proportionality principle}). Let $\mathcal{J}$ be the Van Est isomorphism mapping from $ A^n(\widetilde{M})^G$ to $H_c^n(G,\mathbb{R})$. Then
\[
\|M\|=\dfrac{\mathrm{Vol}(M)}{\|\mathcal{J}\big(\omega_{\widetilde{M}}\big)\|_{\infty}}.
\]
We establish the desired inequality as following.
\begin{theorem}
\label{Thm: norm of omega}Let $\omega_{\mathbb{H}^{2}\times\mathbb{H}^{2}\times\mathbb{H}^{2}}\in %
A^{6}(\mathbb{H}^{2}\times\mathbb{H}^{2}\times\mathbb{H}^{2},\mathbb{R})$ be the Riemannian volume %
form on $\mathbb{H}^{2}\times\mathbb{H}^{2}\times\mathbb{H}^{2}$. Let $\mathcal{J}$ be the Van Est %
isomorphism mapping from $A^{6}(\mathbb{H}^{2}\times\mathbb{H}^{2}\times\mathbb{H}^{2},\mathbb{R})$ to $H_{c}^{6}(\mathrm{PSL}_{2}\mathbb{R}\times\mathrm{PSL}_{2}\mathbb{R}\times\mathrm{PSL}_{2}\mathbb{R},\mathbb{R})$. Then 
\[
 \Vert \mathcal{J}(\omega_{\mathbb{H}^{2}\times\mathbb{H}^{2}\times\mathbb{H}^{2}})  \Vert _{\infty
}\leq\frac{11}{45}\pi^3.
\]
\end{theorem}
We briefly outline the content of each section. To estimate the upper bound for $\|\mathcal{J}(\omega_{\mathbb{H}^{2}\times\mathbb{H}^{2}\times\mathbb{H}^{2}})\|_{\infty}$, we need an explicit cocycle in class $\mathcal{J}(\omega_{\mathbb{H}^{2}\times\mathbb{H}^{2}\times\mathbb{H}^{2}})$. In Section \ref{Section: rep of volume class} we recall the definition of continuous (bounded) cohomology of Lie groups and select a cocycle $\pi^3\cdot\Theta_\theta$ representing $\mathcal{J}(\omega_{\mathbb{H}^{2}\times\mathbb{H}^{2}\times\mathbb{H}^{2}})$. In Section \ref{Section: Est of norm} we recall another bounded cohomology $H_{\delta}^k\big(C_b^*(\mathbb{T}^3,\mathbb{R})^{(PSL_2\mathbb{R})^3}\big)$ which is isometrically isomorphic to the continuous bounded cohomology of $(PSL_2\mathbb{R})^3$. This allows us to estimate the upper bound of $\|[\Theta_\theta]\|_{\infty}$ by calculating the $l^{\infty}$-norm of a selected cocycle $\Theta$ in $C_b^6(\mathbb{T}^3,\mathbb{R})^{(PSL_2\mathbb{R})^3}$. We show that $\|\Theta\|_{\infty}$ can be computed by reducing it to a programming problem. This proves Theorem \ref{Thm: norm of omega}. Finally, in Section \ref{Section: Conjectures} we give several conjectures related to the simplicial volume of closed manifolds covered by $\big(\mathbb{H}^2\big)^n$.

\section{The volume form on $\mathbb{H}^2\times\mathbb{H}^2\times\mathbb{H}^2$}\label{Section: rep of volume class}
Before giving a cocycle representing $\mathcal{J}(\omega_{\mathbb{H}^{2}\times\mathbb{H}^{2}\times\mathbb{H}^{2}})$, we recall some definitions first. 
\subsection{The continuous (bounded) cohomology of Lie groups}
Let $G$ be a connected Lie group. For every $k\in\mathbb{N}$, we define the complex $C^{k}(G,\mathbb{R})$ by
\[
C^{k}(G,\mathbb{R})=\left\{f:G^{k+1}\rightarrow \mathbb{R}\ \middle|\ %
\text{f is continuous} \right\}.
\]
Define the coboundary $d^k :C^{k}(G,\mathbb{R})\rightarrow%
C^{k+1}(G,\mathbb{R})$ by
\[
(d^k f)(g_{0},...,g_{k+1})=%
\sum_{i=0}^{k+1}(-1)^{i}f(g_{0},...,\hat{g}_{i},...,g_{k+1}),
\]
for all $f$ in $C^{k}(G,\mathbb{R})$ and all $(g_0,...,g_{k+1})$ in $G^{k+2}$. Define the $G$-action on $C^{k}(G,\mathbb{R})$ by
\[
(h\cdot f)(g_{0},...,g_{k})=f(h^{-1}g_{0},...,h^{-1}g_{k}),
\]
for all $f$ in $C^{k}(G,\mathbb{R})$, all $h\in G$ and all $(g_0,...,g_k)\in G^{k+1}$. Let $C^{k}(G,\mathbb{R})^G$ be the $G$-invariant elements in $C^{k}(G,\mathbb{R})$. Set $C_{b}^{k}(G,\mathbb{R})$ to be the subspace of $C^{k}(G,\mathbb{R})$ consisting all bounded functions and $d_b^k=d^k|_{C_b^k(G,\mathbb{R})}$. Restrict the $G$-action to $C_{b}^{k}(G,\mathbb{R})$. Then $C_{b}^{k}(G,\mathbb{R})^G\coloneqq C_{b}^{k}(G,\mathbb{R})\cap C^{k}(G,\mathbb{R})^G$ is the subspace of $C_{b}^{k}(G,\mathbb{R})$ with $G$-invariant elements. Therefore $\big(C^{k}(G,\mathbb{R})^{G},d\big)$ and $\big(C_{b}^{k}(G,\mathbb{R})^{G},d_b\big)$ induce continuous cohomology $H_{c}^{k} (G,\mathbb{R} )$ of $G$ and the continuous bounded cohomology $H_{cb}^{k} (G,\mathbb{R} )$ of $G$, respectively. For all $f\in C_{b}^{k} (G,\mathbb{R} )$, let $\|f\|_{\infty}$ be the $l^{\infty}$-norm of $f$. This induces the semi-norms of $H_{cb}^{k} (G,\mathbb{R} )$ and $H_{c}^{k} (G,\mathbb{R} )$, both of which we still denote by $\|\cdot\|_{\infty}$.

Here we use a form of Van Est isomorphism introduced in \cite{dupont1976simplicial}. Let $G$ be a Lie group, $K<G$ be a maximal compact subgroup of $G$ and %
$X=G/K$ be the associated symmetric space. For $k\in\mathbb{N}$, define $\mathrm{A}^{k}(X,\mathbb{R})$ to be the set of differential $k$-forms. The Lie group $G$ acts on $\mathrm{A}^{k}(X,\mathbb{R})$ by the pullbacks. Let $\mathrm{A}^{k} (X,\mathbb{R} )^{G}$ be the subspace of $\mathrm{A}^{k} (X,\mathbb{R} )$ with $G$-invariant elements. Denote the Van Est isomorphism by $\mathcal{J}:\mathrm{A}^{k} (X,\mathbb{R} )^{G}\stackrel{\cong}{\longrightarrow}%
H_{c}^{k}(G,\mathbb{R})$. 

\subsection{The $2$-form representing $\mathcal{J} (\omega_{\mathbb{H}^{2}}  )$}
Let $\omega_{\mathbb{H}^{2}}\in A^{2} (\mathbb{H}^{2},\mathbb{R}  )$ be the volume form on $\mathbb{H}^2$. View $\mathbb{H}^2$ as the upper half-plane $\big\{(x,y)\in\mathbb{R}^2 | y>0\big\}\subset\mathbb{C}$. Then $PSL_{2}\mathbb{R}$ acts on $\mathbb{R}^2$ by the Möbius transformations. This action can be restricted to $\big\{(x,0)\in\mathbb{R}^2\big\}$. Notice that through $z\mapsto\dfrac{z-i}{z+i}$ the upper half-plane is identified with the unit disc $\{|z|<1\}\subset\mathbb{C}$ and the $PSL_{2}\mathbb{R}$-action on $\big\{(x,0)\in\mathbb{R}^2\big\}$ can induce a $PSL_{2}\mathbb{R}$-action on $\mathbb{S}^1$. Define a function $Or$ by
\[
\begin{array}{cccl}
\mathrm{Or}: &  (  \mathbb{S}^{1}  )  ^{3} & \longrightarrow & \mathbb{R}\\
& (\theta_{0},\theta_{1},\theta_{2}) & \longmapsto & \begin{cases}
+1 & \text{if }\theta_{0},\theta_{1},\theta_{2}\text{ are distinct and positively oriented,}\\
-1 & \text{if }\theta_{0},\theta_{1},\theta_{2}\text{ are distinct and negatively oriented,}\\
0 & \text{if }\theta_{0},\theta_{1},\theta_{2}\text{ are not distinct}.
\end{cases}
\end{array}
\]
We fix a point $\theta$ in $\mathbb{S}^1$. Define a cocycle $\mathrm{Or}_{\theta}$ in $C_b^{2}(PSL_{2}\mathbb{R},\mathbb{R})^{PSL_{2}\mathbb{R}}$ by
\[
\mathrm{Or}_{\theta} (g_{0},g_{1},g_{2}  )=\mathrm{Or} (g_{0}\theta,g_{1}\theta,g_{2}\theta  )  %
\]
for all $(g_{0},g_{1},g_{2})$ in $ (PSL_{2}\mathbb{R}  )^3$. Then it is easy to check that %
$\mathcal{J} (\omega_{\mathbb{H}^{2}}  )=\pi\cdot [Or_{\theta}  ]$. 

\subsection{The $6$-form representing $\mathcal{J} (\omega_{\mathbb{H}^{2}\times\mathbb{H}^{2}\times\mathbb{H}^{2}}  )$}

Recall the definition of cup product. Take a $p$-cochain $c^p$ in $C^p\big( (PSL_2\mathbb{R})^3,\mathbb{R}\big)$ and a $q$-cochain $c^q$ in $C^q\big( (PSL_2\mathbb{R}  )^3,\mathbb{R}\big)$. The cup product $c^p\cup c^q$ is defined by
\[
 (c^p\cup c^q  ) (g_0,...,g_{p+q}  )=c^p (g_0,...,g_{p}  )\cdot c^q (g_p,...,g_{p+q}  ),
\]
for all $g_{i}= (g_{i}^{1},g_{i}^{2},g_{i}^{3}  )$ in $ (PSL_{2}\mathbb{R}  )^{3}$, $i=0,...,p+q$. This induces cup product for classes.

Recall that the alternation of a $p$-cochain $c^p$ in $C^p\big( (PSL_2\mathbb{R}  )^3,\mathbb{R}\big)$ is 
\[
\mathrm{Alt}(c^p) (g_0,...,g_{p}  )=\dfrac{1}{(p+1)!}\sum_{\sigma\in Sym(p+1)}sign(\sigma)c^p (g_{\sigma(0)},...,g_{\sigma(p)}  )
\]
for all $g_{i}$ in $ (PSL_{2}\mathbb{R}  )^{3}$, $i=0,...,p$. Note that for a $p$-cocycle $f$ we have $[\mathrm{Alt}(f)]=[f]$ and $\|\mathrm{Alt}(f)\|_{\infty}\leq\|f\|_{\infty}$.

Note Let $p^{\mathbb{H}}_{i}$ be the $i$-th projection from $\mathbb{H}^2\times\mathbb{H}^2\times\mathbb{H}^2$ to $\mathbb{H}^2$ for $i=1,2,3$. Let $p^{PSL_2\mathbb{R}}_{i}$ be the $i$-th projection from $PSL_2\mathbb{R}\times PSL_2\mathbb{R}\times PSL_2\mathbb{R}$ to $PSL_2\mathbb{R}$ for $i=1,2,3$. Let $p^{\mathbb{T}}_{i}$ be the $i$-th projection from $\mathbb{T}^3\times\mathbb{T}^3\times\mathbb{T}^3$ to $\mathbb{T}^3$ for $i=1,2,3$, where $\mathbb{T}^3$ is $\mathbb{S}^1\times\mathbb{S}^1\times\mathbb{S}^1$. 

Define a function $\Theta: (\mathbb{T}^3  )^7\rightarrow\mathbb{R}$ by
\begin{equation}\label{def of Theta}
\begin{split}
    \Theta (\theta_{0},...,\theta_{6}  )&=%
    \mathrm{Alt}\big( (p_1^{\mathbb{T}}  )^*(Or)\cup (p_2^{\mathbb{T}}  )^*(Or)\cup (p_3^{\mathbb{T}}  )^*(Or) \big) (\theta_{0},...,\theta_{6}  )\\
    &=\dfrac{1}{7!}\sum_{\sigma\in Sym (7  )}sign (\sigma  )%
    \prod_{i=1}^{3}Or (\theta_{2i-2}^{i},\theta_{2i-1}^{i},\theta_{2i}^{i}  )
\end{split}
\end{equation}
for all $\theta_{i}= (\theta_{i}^{1},\theta_{i}^{2},\theta_{i}^{3}  )$ in $\mathbb{T}^{3}$, $i=0,...,6$. Let $ (PSL_{2}\mathbb{R}  )^{3}$ acts on $\mathbb{T}^{3}$ diagonally. Fix a point $\theta=(\theta^1,\theta^2,\theta^3)$ in $\mathbb{T}^3$. Define a $6$-form $\Theta_{\theta}$ in $C^6\big( (PSL_2\mathbb{R}  )^3,\mathbb{R}\big)$ by
\[
        \Theta_{\theta} (g_{0},...,g_{6}  )=\Theta (g_{0}\theta,...,g_{6}\theta  )
\]
for all $g_{i}= (g_{i}^{1},g_{i}^{2},g_{i}^{3}  )$ in $ (PSL_{2}\mathbb{R}  )^{3}$, $i=0,...,6$.

\begin{proposition}
$\mathcal{J} (\omega_{\mathbb{H}^{2}\times\mathbb{H}^{2}\times\mathbb{H}^{2}}  )=[\pi^3\cdot\Theta_\theta]$.
\end{proposition}
\begin{proof}
By (\ref{def of Theta}), we have
\begin{equation}\notag
    \begin{split}
        [\pi^3\cdot\Theta_\theta]&=\pi^3\cdot \Big[ Alt\big( (p^{PSL_2\mathbb{R}}_{1}  )^{*}(Or_{\theta^1})\cup  (p^{PSL_2\mathbb{R}}_{2}  )^{*}(Or_{\theta^2})\cup (p^{PSL_2\mathbb{R}}_{3}  )^{*}(Or_{\theta^3}) \big)  \Big]\\
        &=\pi^3\cdot \Big[  (p^{PSL_2\mathbb{R}}_{1}  )^{*}(Or_{\theta^1})\cup  (p^{PSL_2\mathbb{R}}_{2}  )^{*}(Or_{\theta^2})\cup (p^{PSL_2\mathbb{R}}_{3}  )^{*}(Or_{\theta^3}) \Big]\\
        &= \Big[\pi\cdot (p^{PSL_2\mathbb{R}}_{1}  )^{*}(Or_{\theta^1})\Big]\cup \Big[\pi\cdot (p^{PSL_2\mathbb{R}}_{2}  )^{*}(Or_{\theta^2})\Big]\cup \Big[\pi\cdot (p^{PSL_2\mathbb{R}}_{3}  )^{*}(Or_{\theta^3}) \Big]\\
        &=\mathcal{J} \big( (p^{\mathbb{H}}_{1}  )^*(\omega_{\mathbb{H}^2})  \big)\cup\mathcal{J} \big( (p^{\mathbb{H}}_{2}  )^*(\omega_{\mathbb{H}^2})  \big)\cup\mathcal{J} \big( (p^{\mathbb{H}}_{3}  )^*(\omega_{\mathbb{H}^2})  \big)\\
        &=\mathcal{J} \big( (p^{\mathbb{H}}_{1}  )^*(\omega_{\mathbb{H}^2})\wedge (p^{\mathbb{H}}_{2}  )^*(\omega_{\mathbb{H}^2})\wedge (p^{\mathbb{H}}_{3}  )^*(\omega_{\mathbb{H}^2})  \big)\\
        &=\mathcal{J} (\omega_{\mathbb{H}^{2}\times\mathbb{H}^{2}\times\mathbb{H}^{2}}  ).
    \end{split}
\end{equation}
Therefore this proposition is proved.

\end{proof}

\section{Upper bound of $ \|\omega_{\mathbb{H}^2 \times \mathbb{H}^2 \times \mathbb{H}^2}  \|_{\infty}$\label{Section: Est of norm}}

Denote $ (PSL_{2}\mathbb{R}  )^3$ by $G$. For the convenience of calculation, we introduce another complex $\big(C_{b}^{k} (\mathbb{T}^{3},\mathbb{R}  )^{G},\delta \big)$. For every $k\in\mathbb{N}$, define $C_{b}^{k} (\mathbb{T}^{3},\mathbb{R}  )$ by
\[
C_{b}^{k} (\mathbb{T}^{3},\mathbb{R}  )=\left\{f: (\mathbb{T}^{3}  )^{k+1} \rightarrow\mathbb{R}%
\ \middle |\ \text{f is continous, measurable and bounded} \right\}.
\]
The coboundary $\delta:C_{b}^{k} (\mathbb{T}^{3},\mathbb{R})\rightarrow%
C_{b}^{k+1} (\mathbb{T}^{3},\mathbb{R}  )$ is defined by
\[
 (\delta f  ) (\theta_{0},...,\theta_{k+1}  )=%
\sum_{i=0}^{k+1} (-1  )^i f (\theta_{0},...,\hat{\theta}_{i},...,\theta_{k+1}  )
\]
for all $f$ in $C_{b}^{k} (\mathbb{T}^{3},\mathbb{R}  )$ and all $\theta_{i}$ in $\mathbb{T}^3$, $i=0,...,k+1$.
Define the action of $G$ on $C_{b}^{k} (\mathbb{T}^{3},\mathbb{R}  )$ by
\[
 (g \cdot f  )  (\theta_{0},...,\theta_{k}  )=f (g^{-1}\theta_{0},...,g^{-1}\theta_{k}  ).
\]
Denote the $G$-invariant elements in $C_{b}^{k} (\mathbb{T}^{3},\mathbb{R}  )$ by $C_{b}^{k} (\mathbb{T}^{3},\mathbb{R}  )^{G}$. Then the complex $\big(C_{b}^{k} (\mathbb{T}^{3},\mathbb{R}  )^{G},\delta\big)$ induces a cohomology $H_{\delta}^{k}\big(C_{b}^{*} (\mathbb{T}^{3},\mathbb{R}  )^G \big)$. Let $\|f\|_{\infty}$ be the $l^\infty$-norm of $f$ in $C_{b}^{k} (\mathbb{T}^{3},\mathbb{R}  )$. This induces the semi-norm of $H_{\delta}^{k}\big(C_{b}^{*} (\mathbb{T}^{3},\mathbb{R}  )^G\big)$, which we still denote by $\|\cdot\|_{\infty}$.

Let $P$ be the subgroup of $PSL_2\mathbb{R}$ defined as below.
\[
P=\left\{\begin{pmatrix} a & 0 \\ c & a^{-1} \end{pmatrix}\in SL_{2}\mathbb{R}\ \middle |\ a\in\mathbb{R}\setminus\{0\},\ c\in\mathbb{R}\right\}\bigg/%
 \{\pm 1  \}.
\]
Notice that $\mathbb{T}^3= (PSL_{2}\mathbb{R}  )^3\big/P^3$ and $P$ is amenable. Using  \cite[Corollary 7.5.9.]{monod2001continuous}, $\big(H_{cb}^{k} (G,\mathbb{R}  ),\|\ \|_{\infty}\big)$ is isometrically isomorphic to $ \big(H_{\delta}^{k} (C_{b}^{*} (\mathbb{T}^{3},\mathbb{R}  )^G   ),\|\ \|_{\infty}\big)$. Note that this isomorphism maps $ [\Theta_{\theta}  ]$ to $ [\Theta  ]$. Hence
\[
     \|\mathcal{J} (\omega_{\mathbb{H}^2 \times \mathbb{H}^2 \times \mathbb{H}^2}  )  \|_{\infty}%
    =%
    \inf \left\{ \|f  \|_{\infty}\ \middle|\ f\in [\Theta_{\theta}  ] \right\}%
    \leq%
     \|\Theta  \|_{\infty}.
\]
As a conclusion, to estimate the upper bound of $ \|\omega_{\mathbb{H}^2 \times \mathbb{H}^2 \times \mathbb{H}^2}  \|_{\infty}$ we need to calculate $\|\Theta\|_{\infty}$. 
\begin{proposition}
\label{Prop: norm of Theta}$ \Vert \Theta  \Vert _{\infty
}=\dfrac{11}{45}$.
\end{proposition}

Before using algorithm to compute the exact value of $\|\Theta\|_\infty$ we need to simplify the formula (\ref{def of Theta}). Otherwise the computation progress using algorithm will take forever. For all %
$ (\theta_{0},...,\theta_{6}  )$ in $ (\mathbb{T}^3  )^7$, we have
\begin{equation}\label{firt simplify of theta}
    \begin{split}
        \Theta (\theta_{0},...,\theta_{6}  )&=%
        \dfrac{1}{7!}\sum_{\sigma\in Sym (7  )}sign (\sigma  )%
        Or\Big(\theta_{\sigma (0  )}^{1},\theta_{\sigma (1  )}^{1},%
        \theta_{\sigma (2  )}^{1}\Big)\\
        &\ \ \ \ \ \ \ \ \ \ \ \ \ \ \ \ \ \ \
        Or\Big(\theta_{\sigma (2  )}^{2},\theta_{\sigma (3  )}^{2},%
        \theta_{\sigma (4  )}^{2}\Big)%
        Or\Big(\theta_{\sigma (4  )}^{3},\theta_{\sigma (5  )}^{3},%
        \theta_{\sigma (6  )}^{3}\Big)\\
        &=%
        \dfrac{1}{7!}\sum_{\sigma\in Sym (7  )}sign (\sigma  )%
        Or\Big(\theta_{\sigma (0  )}^{1},\theta_{\sigma (1  )}^{1},%
        \theta_{\sigma (2  )}^{1}\Big)\\
        &\ \ \ \ \ \ \ \ \ \ \ \ \ \ \ \ \ \ \
        Or\Big(\theta_{\sigma (0  )}^{2},\theta_{\sigma (3  )}^{2},%
        \theta_{\sigma (4  )}^{2}\Big)%
        Or\Big(\theta_{\sigma (4  )}^{3},\theta_{\sigma (5  )}^{3},%
        \theta_{\sigma (6  )}^{3}\Big)\\
        &=%
        \dfrac{4}{7!}\sum_{\substack{\sigma\in Sym (7  )\\%
        \sigma(1)<\sigma(2),\ \sigma(5)<\sigma(6)}}sign (\sigma  )%
        Or\Big(\theta_{\sigma (0  )}^{1},\theta_{\sigma (1  )}^{1},%
        \theta_{\sigma (2  )}^{1}\Big)\\
        &\ \ \ \ \ \ \ \ \ \ \ \ \ \ \ \ \ \ \
        Or\Big(\theta_{\sigma (0  )}^{2},\theta_{\sigma (3  )}^{2},%
        \theta_{\sigma (4  )}^{2}\Big)%
        Or\Big(\theta_{\sigma (4  )}^{3},\theta_{\sigma (5  )}^{3},%
        \theta_{\sigma (6  )}^{3}\Big).
    \end{split}
\end{equation}

View $\mathbb{S}^1$ as the quotient space $[0,1]/\sim_{f}$, where $f(x)=e^{2\pi i x}$. Denote $\theta_{i}$ by $\big( [x_{i}  ], [y_{i}  ], [z_{i}  ]\big)$, where $x_{i},y_{i},z_{i}\in[0,1]$, for $i=0,...,6$. To prove %
Proposition \ref{Prop: norm of Theta}, we need the following lemmas.

\begin{lemma}\label{norm when donnot have same componet for x,y,z}
If the elements in  $ \{\theta_{j}^{i}  \}_{j=0}^{6}$ are pairwise distinct for $i=1,2,3$,
\begin{equation}\notag
     |\Theta (\theta_{0},...,\theta_{6}  )  |\leq\dfrac{11}{45}.
\end{equation}
\end{lemma}
\begin{proof}
We can assume that $x_{0}<...< x_{6}<1$ and $\Theta$ is alternating. We notice that for all $\theta_{i}=[x_{i}]\in\mathbb{S}^1$, where $i=0,1,2$,
\[
Or (\theta_{0},\theta_{1},\theta_2  )=sign(x_0,x_1,x_2).
\]
Here we abuse the notation $sign$. If the elements of $\{x_i\}_{i=0}^k$ are pairwise distinct, assuming $x_{0}<...<x_{k}$, we define $sign(x_{i_0},...,x_{i_k})$ to be $sign(i_0,...,i_k)$. If the elements of $\{x_i\}_{i=0}^k$ are not pairwise distinct, we define $sign(x_{i_0},...,x_{i_k})$ to be $0$. Therefore we only need to consider the order of $\{y_i\}$ and $\{z_i\}$.
Moreover, notice that take $g= (g_{1},g_2,g_3  )$ in $ (PSL_2\mathbb{R}  )^3$ with $g_i$ being rotation or reflection of $\mathbb{S}^1$ for $i=1,2,3$. Then
\begin{equation}\label{g pres norm}
   | (g\cdot\Theta  ) (\theta_{0},...,\theta_{6}  )  |=%
 |\Theta (\theta_{0},...,\theta_{6}  )  |
\end{equation}
for all $\theta_i\in\mathbb{T}^3$, $i=0,...,6$.
Thus we can assume $y_0=z_0=0$, $y_1<y_2$ and $z_1<z_2$. Therefore we only need to consider $360$ possible orders of $\{y_i\}_{i=0}^6$, as well as of $\{z_i\}_{i=0}^6$. We denote these possible orders by a $7\times 360$ matrix $P$, where each column represents a possible order. For example, a column $(2,3,5,1,4,6,7)^{t}$ represents $y_3<y_0<y_1<y_4<y_2<y_5<y_6$ or $z_3<z_0<z_1<z_4<z_2<z_5<z_6$. Let $S$ be a $7\times1260$ matrix, where each column represents a permutation $\sigma$ in $Sym(1,...,7)$ satisfying $\sigma(2)<\sigma(3)$ and $\sigma(6)<\sigma(7)$. To further simplify the computation progress, we define two $1260\times 360$ matrices $P_{y360}$ and $P_{z360}$ by
\[
 (P_{y360}  )_{i,j}=sign\big(P_{(S_{1,i}),j},P_{(S_{4,i}),j},P_{(S_{5,i}),j}\big)
\]
and
\[
 (P_{z360}  )_{i,j}=sign\big(P_{(S_{5,i}),j},P_{(S_{6,i}),j},P_{(S_{7,i}),j}\big).
\]
We define a $1260\times1$ matrix $P_{sign}$ by
\[
 (P_{sign}  )_{i,1}=sign (S_{1,i},...,S_{7,i}  )\cdot sign (S_{1,i},S_{2,i},S_{3,i}  ).
\]
We apply the input values $P_y=P_{y360}$, $P_{z}=P_{z360}$ and $P_{sign}=P_{sign}$ to Algorithm \ref{alg for outcome} and denote the corresponding output by $O_{360}$. Then we get all $360\times360$ possible values (possibly repeated) $O_{360}$ of $\Theta$.

\begin{breakablealgorithm}
\caption{possible values of $\Theta(\theta_0,...,\theta_6)$.}
\label{alg for outcome}
\begin{algorithmic}[1]
		\begin{footnotesize} 
			\REQUIRE Matrices $P_{y}$, $P_{z}$ and $P_{sign}$;
            \ENSURE Matrix $O$;
            \STATE Set p and q to be the number of columns of $P_y$ and $P_{z}$ respectively;
            \STATE Set $O$ to be a $p\times q$ zero matrix;
            \FOR{each $i=1,...,p$}
            \FOR{each $j=1,...,q$}
            \FOR{each $l=1,...,1260$}
            \STATE $ (O  )_{i,j}= (O  )_{i,j}+\dfrac{1}{1260} (P_{sign}  )_{l,1}\cdot (P_{y360}  )_{l,i}\cdot (P_{z360}  )_{l,j}$;
            \ENDFOR
            \ENDFOR
            \ENDFOR
		\end{footnotesize}
	\end{algorithmic}
\end{breakablealgorithm}
Hence in this case $ |\Theta (\theta_{0},...,\theta_{6}  )  |\leq\dfrac{11}{45}$. For $(\theta_0,...,\theta_6)$ defined as in Figure \ref{fig:maximal value 3}, $\Theta$ attains its maximal volume $\dfrac{11}{45}$.\\
\begin{figure}[H]
    \begin{tikzpicture}
    
    \matrix[
       row sep=0.3cm,column sep=0.5cm] {
        \draw[dashed, gray] (0,0) -- (0:1);
    \draw[dashed, gray] (0,0) -- (360/7:1);
    \foreach \x [count=\p] in {0,...,6} {
        \node[shape=circle,fill=black, scale=0.3] (\p) at (\x*360/7:1) {};};
        \foreach \x [count=\p] in {0,...,6} 
    {
    
        \draw (\x*360/7:1.4) node {$[x_{\x}]$};
    };  
    \draw (1,0) arc (0:360:1);&
    \draw[dashed, gray] (0,0) -- (0:1);
    \draw[dashed, gray] (0,0) -- (720/7:1);
    ;
    \foreach \x [count=\p] in {0,...,6} {
        \node[shape=circle,fill=black, scale=0.3] (\p) at (\x*720/7:1) {};};
        \foreach \x [count=\p] in {0,...,6} 
    {
    
        \draw (\x*720/7:1.4) node {$[y_{\x}]$};
    };  
    \draw (1,0) arc (0:360:1);&
    \draw[gray,dashed] (0,0) -- (0:1);
    \draw[gray,dashed] (0,0) -- (1080/7:1);
    \foreach \x [count=\p] in {0,...,6} {
        \node[shape=circle,fill=black, scale=0.3] (\p) at (\x*360/7:1) {};};
        \foreach \x [count=\p] in {0,...,6} 
    {
    
        \draw (\x*1080/7:1.4) node {$[z_{\x}]$};
    };  
    \draw (1,0) arc (0:360:1);&\\
    };
\end{tikzpicture}

    \caption{}
    \label{fig:maximal value 3}
\end{figure}

\end{proof}

\begin{lemma}\label{x3points}
If $ \{\theta_j^1  \}_{j=0}^6$ has exactly $3$ distinct points in $\mathbb{S}^1$,
\[
 |\Theta (\theta_{0},...,\theta_{6}  )  |<\dfrac{11}{45}.
\]
\end{lemma}
\begin{proof}
By (\ref{firt simplify of theta}), we have 
\begin{equation}\notag
\begin{split}
    \Theta (\theta_{0},...,\theta_{6}  )&=%
    \dfrac{1}{7!}\sum_{\sigma\in Sym (7  )}sign (\sigma  )%
    Or\big([x_{\sigma (0  )}],[x_{\sigma (1  )}],%
    [x_{\sigma (2  )}]\big)\\
    &\ \ \ \ \ \ \ \ \ \ \ \ \ Or\big([y_{\sigma (0  )}],[y_{\sigma (3  )}],%
    [y_{\sigma (4  )}]\big)%
    Or\big([z_{\sigma (4  )}],[z_{\sigma (5  )}],%
    [z_{\sigma (6  )}]\big)\\
    &=\dfrac{1}{7!}\sum_{\substack{a<b<c\\\{a,b,c\}\in\{0,...,6\}}}%
    2\cdot Or\big([x_a],[x_b],[x_c]\big)%
    \sum_{\substack{\{i,j,k,l\}=\{0,...,6\}\backslash\{a,b,c\}\\k<l}}\\
    &\ \ \ \ 2\cdot sign(a,b,c,i,j,k,l)\Big[Or\big([y_a],[y_i],[y_j]\big)Or\big([z_j],[z_k],[z_l]\big)+\\
    &\ \ \ \ Or\big([y_b],[y_i],[y_j]\big)Or\big([z_j],[z_k],[z_l]\big)+\\
    &\ \ \ \ Or\big([y_c],[y_i],[y_j]\big)Or\big([z_j],[z_k],[z_l]\big)\Big].
\end{split}
\end{equation}
\begin{claim}\label{2/3claim}
For three fixed element $a,b,c\in\{0,...,6\}$ satisfying $a<b<c$ there are at least $\dfrac{1}{3}$ of the items in
\begin{equation}\notag
    \begin{split}
    &Or\big([x_a],[x_b],[x_c]\big)%
    \sum_{\substack{\{i,j,k,l\}=\{0,...,6\}\backslash\{a,b,c\}\\k<l}}2\cdot sign(a,b,c,i,j,k,l)\\
    &\ \ \ \ \ \ \ \ \ \ \Big[Or\big([y_a],[y_i],[y_j]\big)Or\big([z_j],[z_k],[z_l]\big)+Or\big([y_b],[y_i],[y_j]\big)Or\big([z_j],[z_k],[z_l]\big)+\\
    &\ \ \ \ \ \ \ \ \ \ Or\big([y_c],[y_i],[y_j]\big)Or\big([z_j],[z_k],[z_l]\big)\Big]
    \end{split}
\end{equation}
that vanish or cancel with each other.
\end{claim}
\begin{claimproof}

Fix $\{a,b,c\}$ in $\{0,...,6\}$, and denote $\{0,...,6\}\backslash\{a,b,c\}$ by $\{i_0,...,i_3\}$. We can assume that $z_0\leq z_1\leq...\leq z_6$ as $\Theta$ is alternating. Assuming $i_0<...<i_3$, we have $z_{i_0}\leq...\leq z_{i_3}$. 

If $\theta_{i_0}^3$, $\theta_{i_1}^3$, $\theta_{i_2}^3$ and $\theta_{i_3}^3$ are four pairwise distinct points in $\mathbb{S}^1$, we have 
\begin{equation}\label{1/3 vanishing}
    \begin{split}
    &\ \ \ Or\big([x_a],[x_b],[x_c]\big)%
    \sum_{\substack{\{i,j,k,l\}=\{0,...,6\}\backslash\{a,b,c\}\\k<l}}2\cdot sign(a,b,c,i,j,k,l)\\
    &\ \ \ \ \ \ \ \ \ \ \ \ \ \ \ \ \ \ \ \ \ \ \ \ \ \ \ \ \ \ \ \ \ \ \ \ \ \ \ \ \ \ \ \ \ Or\big([y_a],[y_i],[y_j]\big)Or\big([z_j],[z_k],[z_l]\big)\\
    &=Or\big([x_a],[x_b],[x_c]\big)\sum_{\substack{\tau\in Sym(i_0,...,i_3)\\\tau(i_2)<\tau(i_3)}}%
    2\cdot sign(a,b,c,\tau)sign(\tau(i_1),\tau(i_{2}),\tau(i_3))\\
    &\ \ \ \ \ \ \ \ \ \ \ \ \ \ \ \ \ \ \ \ \ \ \ \ \ \ \ \ \ \ \ \ \ \ \ \ \ \  Or\big([y_{a}],[y_{\tau(i_0)}],[y_{\tau(i_1)}]\big)\\
    &=Or\big([x_a],[x_b],[x_c]\big)\sum_{\substack{\tau\in Sym(i_0,...,i_3)\\\tau(i_2)<\tau(i_3)}}2\cdot(-1)^{a+b+c+\alpha(i_0)}Or\big([y_a],[y_{\tau(i_0)}],[y_{\tau(i_1)}]\big),
    \end{split}
\end{equation}
where for $\sigma(i_0)=i_{p}$, $\alpha\big(\sigma(i_0)\big)=p$. Notice that the items corresponding to $\sigma= (a,b,c,i_3,i_1,i_0,i_2  )$ cancels the items corresponding to $\sigma= (a,b,c,i_1,i_3,i_0,i_2  )$ and the items corresponding to $\sigma= (a,b,c,i_2,i_0,i_1,i_3  )$ cancels the items corresponding to $\sigma= (a,b,c,i_0,i_2,i_1,i_3  )$. Therefore at least $\dfrac{1}{3}$ of the items in (\ref{1/3 vanishing}) cancel with each other.

If $\theta_{i_0}^3$, $\theta_{i_1}^3$, $\theta_{i_2}^3$ and $\theta_{i_3}^3$ have at least two points coincide with each other, we have $\dfrac{\binom{1}{2}}{\binom{3}{4}}=\dfrac{1}{2}$ of the items in (\ref{1/3 vanishing}) vanish. The same goes for 
\begin{equation}\notag
    \begin{split}
        &\ \ \ Or\big([x_a],[x_b],[x_c]\big)%
    \sum_{\substack{\{i,j,k,l\}=\{0,...,6\}\backslash\{a,b,c\}\\k<l}}2\cdot sign(a,b,c,i,j,k,l)\\
    &\ \ \ \ \ \ \ \ \ \ \ \ \ \ \ \ \ \ \ \ \ \ \ \ \ \ \ \ \ \ \ \ \ \ \ \ \ \ \ \ \ \ \ \ \ Or\big([y_b],[y_i],[y_j]\big)Or\big([z_j],[z_k],[z_l]\big)
    \end{split}
\end{equation}
and
\begin{equation}\notag
    \begin{split}
        &\ \ \ Or\big([x_a],[x_b],[x_c]\big)%
    \sum_{\substack{\{i,j,k,l\}=\{0,...,6\}\backslash\{a,b,c\}\\k<l}}2\cdot sign(a,b,c,i,j,k,l)\\
    &\ \ \ \ \ \ \ \ \ \ \ \ \ \ \ \ \ \ \ \ \ \ \ \ \ \ \ \ \ \ \ \ \ \ \ \ \ \ \ \ \ \ \ \ \ Or\big([y_c],[y_i],[y_j]\big)Or\big([z_j],[z_k],[z_l]\big).
    \end{split}
\end{equation}
Hence we get the required result.
\end{claimproof}
Now that we have these facts, we can prove this lemma in 4 cases.
\begin{case}\label{case1}
There are $1$, $1$ and $5$ points in $\{\theta_j^1\}_{j=0}^6$ that respectively have the same values in $\mathbb{S}^1$.
\end{case}
We notice that $Or\big([x_0],[x_1],[x_2]\big)=0$ when there exists $i\neq j$ such that $[x_i]=[x_j]$. Hence there are only $\dfrac{\binom{1}{5}}{\binom{3}{7}}=\dfrac{1}{7}<\dfrac{11}{45}$ of the items in (\ref{def of Theta}) which may not vanish. Therefore this case stands.
\begin{case}
There are $1$, $2$ and $4$ points in $\{\theta_j^1\}_{j=0}^6$ that respectively have the same values in $\mathbb{S}^1$.
\end{case}
Similarly to Case \ref{case1}, there are only $\dfrac{\binom{1}{2}\cdot\binom{1}{4}}{\binom{3}{7}}=\dfrac{8}{35}<\dfrac{11}{45}$ of the items in (\ref{def of Theta}) which may not vanish. Therefore this case stands.
\begin{case}
There are $1$, $3$ and $3$ points in $\{\theta_j^1\}_{j=0}^6$ that respectively have the same values in $\mathbb{S}^1$.
\end{case}\label{case3}
According to Claim \ref{2/3claim}, there are only $\dfrac{\binom{1}{3}\cdot\binom{1}{3}}{\binom{3}{7}}\times\dfrac{2}{3}=\dfrac{6}{35}<\dfrac{11}{45}$ of the items in (\ref{def of Theta}) which may not vanish or cancel with each other. Therefore this case stands.
\begin{case}
There are $2$, $2$ and $3$ points in $\{\theta_j^1\}_{j=0}^6$ that respectively have the same values in $\mathbb{S}^1$.
\end{case}
Similarly to Case \ref{case3}, there are only $\dfrac{\binom{1}{2}\cdot\binom{1}{2}\cdot\binom{1}{3}}{\binom{3}{7}}\times\dfrac{2}{3}=\dfrac{8}{35}<\dfrac{11}{45}$ of the items in (\ref{def of Theta}) which may not vanish or cancel with each other. Therefore this case stands.

In conclusion, this lemma is proved.

\end{proof}
\begin{lemma}\label{x4points}
If $ \{\theta_j^1  \}_{j=0}^6$ has at least $4$ distinct points in $\mathbb{S}^1$,
\[
 |\Theta (\theta_{0},...,\theta_{6}  )  |<\dfrac{11}{45}.
\]
\end{lemma}
\begin{proof}
We prove this lemma in $3$ cases.
\begin{case}
There are $1$, $1$, $1$ and $4$ points in $\{\theta_j^1\}_{j=0}^6$ that respectively have the same values in $\mathbb{S}^1$.
\end{case}
When there are $2$ points in $\{\theta_j^2\}_{j=0}^6$ that have the same value in $\mathbb{S}^1$, assuming that $x_0=...= x_3<x_4<x_5<x_6<1$, the worst case is $y_4=y_5$. Taking a closer look at the proof of Claim \ref{2/3claim}, we have that there are at most  
\[
\dfrac{(1+4)\times\frac{2}{3}+(2+4)\times\frac{2}{3}\times\frac{8+2}{12}}{\binom{3}{7}}=\dfrac{2}{9}<\dfrac{11}{45}
\]
of the items in (\ref{def of Theta}) that may not vanish or cancel with each other. Therefore we can assume that $\{\theta_j^i\}_{j=0}^6$ are pairwise distinct for $i=2$ and $3$, respectively. We use Algorithm \ref{alg for outcome} to prove this. Let $(x_1,...,x_7)$ be $(1,1,1,1,2,3,4)$. Define a $1260\times 1$ matrix $P_{sign1114}$ by
\[
 (P_{sign1114}  )_{i,1}=sign(S_{1,i},...,S_{7,i})\cdot sign(x_{\scriptscriptstyle S_{1,i}},x_{\scriptscriptstyle S_{2,i}},x_{\scriptscriptstyle S_{3,i}}).
\]
Here we input $P_y=P_{y360}$, $P_{z}=P_{z360}$ and $P_{sign}=P_{sign1114}$. Denote the corresponding output by $O$. The maximal value of entries of $O$ is $\dfrac{2}{15}$ which is smaller than $\dfrac{11}{45}$. Therefore this case stands.
\begin{case}
There are $1$, $1$, $2$ and $3$ points in $\{\theta_j^1\}_{j=0}^6$ that respectively have the same values in $\mathbb{S}^1$.
\end{case}
Let $P_2$ be a $7\times \dfrac{7!}{2!}$ (i.e., $7\times 2520$) matrix which each column represents a permutation of $(1,1,2,3,4,5,6)$. Let $P_3$ be a $7\times \dfrac{7!}{2!\times2!}$ (i.e., $7\times 1260$) matrix which each column represents a permutation of $(1,1,2,2,3,4,5)$. Let $P_4$ be a $7\times \dfrac{7!}{2!\times2!}$ (i.e., $7\times 1260$) matrix which each column represents a permutation of $(1,1,2,3,3,4,5)$. Let $P_5$ be a $7\times \dfrac{7!}{3!}$ (i.e., $7\times 840$) matrix which each column represents a permutation of $(1,1,1,2,3,4,5)$. Let $P_6$ be a $7\times \dfrac{7!}{2!\times2!\times 2!}$ (i.e., $7\times 630$) matrix which each column represents a permutation of $(1,1,2,2,3,3,4)$. Let $P_7$ be a $7\times \dfrac{7!}{3!\times2!}$ (i.e., $7\times 420$) matrix which each column represents a permutation of $(1,1,1,2,2,3,4)$. Let $P_8$ be a $7\times \dfrac{7!}{2!\times2!}$ (i.e., $7\times 420$) matrix which each column represents a permutation of $(1,1,1,2,3,3,4)$. Define a $7\times 7710$ matrix $P_r$ by stacking $P$, $P_2$,...,$P_7$ and $P_8$ vertically.
Let $P_{yr}$ be 
\[
 (P_{yr}  )_{k,l}=sign\big( (P_{r}  )_{(S_{1,k}),l}, (P_{r}  )_{(S_{4,k}),l}, (P_{r}  )_{(S_{5,k}),l}\big).
\]
Let $P_{zr}$ be 
\[
 (P_{zr}  )_{k,l}=sign\big( (P_{r}  )_{(S_{5,k}),l}, (P_{r}  )_{(S_{6,k}),l}, (P_{r}  )_{(S_{7,k}),l}\big).
\]
Let $(x_1,...,x_7)$ be $(1,1,1,2,2,3,4)$. Define a $1260\times 1$ matrix $P_{sign1123}$ by
\[
 (P_{sign1123}  )_{i,1}=sign(S_{1,i},...,S_{7,i})\cdot sign(x_{S_{1,i}},x_{S_{2,i}},x_{S_{3,i}}).
\]
Input $P_y=P_{yr}$, $P_{z}=P_{zr}$ and $P_{sign}=P_{sign1123}$ to get the corresponding output $O1123$. Define $(x_1,...,x_7)=(1,1,1,2,3,3,4)$. Define a $1260\times 1$ matrix $P_{sign1213}$ by
\[
 (P_{sign1213}  )_{i,1}=sign(S_{1,i},...,S_{7,i})\cdot sign(x_{S_{1,i}},x_{S_{2,i}},x_{S_{3,i}}).
\]
Input $P_y=P_{yr}$, $P_{z}=P_{zr}$ and $P_{sign}=P_{sign1213}$ to get the corresponding output $O1213$. Notice that the maximal values of  entries of $O1123$ and $O1213$ are both $\dfrac{7}{45}$. Notice that
\begin{equation}\label{change the position of xyz}
    \begin{split}
            \Theta\Big(\big([x_0],[y_0],[z_0]\big),...,\big([x_6],[y_6],[z_6]\big)\Big)%
    &=\Theta\Big(\big([y_0],[x_0],[z_0]\big),...,\big([y_6],[x_6],[z_6]\big)\Big)\\
    &=\Theta\Big(\big([x_0],[z_0],[y_0]\big),...,\big([x_6],[z_6],[y_6]\big)\Big).
    \end{split}
\end{equation}
Therefore this case stands.
\begin{case}
All other cases which are not include above.
\end{case}
Using equation (\ref{g pres norm}) and (\ref{change the position of xyz}), we only need to apply Algorithm \ref{alg for outcome} as below.
Let $(x^i_1,...,x^i_7)$ be $(1,1,2,2,3,3,4)$, $(1,1,1,2,3,4,5)$, $(1,1,2,2,3,4,5)$ and $(1,1,2,3,4,5,6)$ respectively for $i=1,...,4$. Define four $1260\times 1$ matrices $\{P_{sign}^i\}_{i=1}^4$ by
\[
 (P_{sign}^i  )_{j,1}=sign(S_{1,j},...,S_{7,j})\cdot sign\big(x^i_{S_{1,j}},x^i_{S_{2,j}},x^i_{S_{3,j}}\big)
\]
respectively for $i=1,...,4$. Input $P_y=P_{yr}$, $P_{z}=P_{zr}$ and $P_{sign}=P_{sign}^i$ to get the outputs $O^i=O$ respectively for $i=1,...,4$. The maximal values of entries of $O^i$ are $\dfrac{8}{45}$, $\dfrac{8}{45}$, $\dfrac{1}{5}$ and $\dfrac{2}{9}$ for $i=1,...,4$, respectively. Therefore this case stands.
\end{proof}

Consequently, Proposition \ref{Prop: norm of Theta} is proved. Therefore, we get the desired estimation $\|\omega_{\mathbb{H}^2\times\mathbb{H}^2\times\mathbb{H}^2}\|_{\infty}\leq\dfrac{11\pi^3}{45}$. As a conclusion, Theorem \ref{Thm: norm of omega} is proved. 
Furthermore, Algorithm \ref{alg for outcome} can easily be generalized to provide an algorithm to compute the lower bound of closed Riemannian manifolds covered by $\big(\mathbb{H}^2\big)^n$.

\section{Conjectures}\label{Section: Conjectures}
Denote the $4$-cocycle representing $\omega_{\mathbb{H}^2\times\mathbb{H}^2}$ in \cite[Proposition 7]{bucher2008simplicial} by $\Theta'$. Notice that when taking value $(\theta_0,...,\theta_4)=\Big( \big([x_0],[y_0]\big),...,\big([x_4],[y_4]\big)\Big)$ as in Figure \ref{fig:maximal value 2}, $\Theta'$ attains its maximal volume $\dfrac{2}{3}$. 
\begin{figure}[H] 
    \begin{tikzpicture}
    
    \matrix[
       row sep=0.3cm,column sep=0.5cm] {
    \draw[dashed, gray] (0,0) -- (0:1);
    \draw[dashed, gray] (0,0) -- (360/5:1);
    \foreach \x [count=\p] in {0,...,4} {
        \node[shape=circle,fill=black, scale=0.3] (\p) at (\x*360/5:1) {};};
        \foreach \x [count=\p] in {0,...,4} 
    {
    
        \draw (\x*360/5:1.4) node {$[x_{\x}]$};
    };  
    \draw (1,0) arc (0:360:1);&
    \draw[dashed, gray] (0,0) -- (0:1);
    \draw[dashed, gray] (0,0) -- (720/5:1);
    \foreach \x [count=\p] in {0,...,4} {
        \node[shape=circle,fill=black, scale=0.3] (\p) at (\x*720/5:1) {};};
        \foreach \x [count=\p] in {0,...,4} 
    {
    
        \draw (\x*720/5:1.4) node {$[y_{\x}]$};
    };  
    \draw (1,0) arc (0:360:1);&\\
    };
\end{tikzpicture}
\caption{}\label{fig:maximal value 2}
\end{figure}
We notice a pattern comparing Figure \ref{fig:maximal value 3} with \ref{fig:maximal value 2} as following.
\begin{conjecture}
Define a $2n$-form $\Theta_n$ in $C_b^{2n}(\mathbb{T}^n,\mathbb{R})$ as
\begin{equation}\notag
    \Theta_n (\theta_{0},...,\theta_{2n}  )=%
    \dfrac{\pi^n}{(2n+1)!}\sum_{\sigma\in Sym (2n+1  )}sign (\sigma  )%
    \prod_{i=1}^{n}Or\big(\theta_{2i-2}^{i},\theta_{2i-1}^{i},\theta_{2i}^{i}\big)
\end{equation}
for all $\theta_{i}= (\theta_{i}^{1},...,\theta_{i}^{n}  )$ in $\mathbb{T}^{n}$, $i=0,...,2n$. Then $\|\Theta_n\|_{\infty}=\Theta_n (\theta_0,...,\theta_{2n}  )$, where $\theta^k_i=e^{\frac{2(ki)\pi\sqrt{-1}}{2n+1}}$ for $k=1,...,n$ and $i=0,...,2n$.
\end{conjecture}

We already have $\|\omega_{(\mathbb{H}^2)^n}\|_{\infty}\leq\|\Theta_n\|_{\infty}$ and the proportionality principle for locally symmetric spaces. We can ask further that whether this inequality is actually equality.
\begin{conjecture}
Let $M$ be a closed oriented manifold covered by $\big(\mathbb{H}^2\big)^n$. Then the simplicial volume $\|M\|=\dfrac{\mathrm{Vol}(M)}{\|\Theta_n\|_{\infty}}$.
\end{conjecture}
If these two conjecture are true, we can get the exact value of simplicial volume of all closed oriented manifolds covered by $\big(\mathbb{H}^2\big)^n$.

\end{document}